\documentclass[12pt]{article}
\usepackage{latexsym,color,amsmath,amsthm,amssymb,amscd,amsfonts}
\usepackage{tikz}
\usetikzlibrary{arrows}
\usepackage{scalefnt}
\usepackage[font=small,labelfont=bf]{caption}
\usepackage{float}
\usepackage{graphicx} 
\usepackage{dsfont}
\usepackage{amsopn}
\usepackage{relsize}
\usepackage{mathrsfs}
\usepackage{upgreek}

\newtheoremstyle{nonum}{}{}{\itshape}{}{\bfseries}{.}{ }{\thmnote{#3}}

\newtheorem{thm}{Theorem}[section]%[subsection]
\newtheorem*{thm*}{Theorem}

\newtheorem{lem}[thm]{Lemma}

\newtheorem*{definition*}{Definition}

\newtheorem*{rems*}{Remarks}

\theoremstyle{nonum}

\newcommand{\R}{\mathbb R}

\newcommand{\N}{\mathbb N}

\def\Sn{T^n}
\def\K{{\cal K}}

\newcommand{\iprod}[2]{\langle #1,#2 \rangle} % Inner product %

\def\K{{\cal K}}

\def\conv{{\rm conv}}

\begin{document}
\title {Minkowski Symmetrizations of Star Shaped Sets}
\date{}
\author{D.I. Florentin, A. Segal}
\maketitle
\begin{abstract}
We provide sharp upper bounds for the number of symmetrizations required to
transform a star shaped set in $\R^n$ %to a body which is
arbitrarily close (in the Hausdorff metric) to the Euclidean ball.
\end{abstract}

\section{Introduction and results}\label{Sec_Intro}
A non empty compact set $K\subset \R^n$ is called {\em star shaped}
if $x\in K$ implies $[0,x]\subseteq K$. We denote the family of star
shaped sets in $\R^n$ by $\Sn$. Recall that given a set $K$ and a
direction $u \in S^{n-1}$, it's Minkowski symmetral is defined to be
\[
M_u(K) = \frac{K + R_u K}{2},
\] 
where $R_u$ is the reflection with respect to the hyperplane
$u^\perp$.

The Minkowski symmetrization $M_u$ results in a set that is symmetric
with respect to the hyperplane $u^\perp$, thus it is natural to
expect that successive applications of this procedure in different
directions yield a sequence of sets that convergences in some sense
to the Euclidean ball. %(with the same mean width). 
This is indeed known in the case where $K$ is convex.
Moreover, there are estimates regarding the convergence rate.
Bourgain, Lindenstrauss and Milman \cite{BLM} obtained the first
quantitative estimate for the convergence rate of Minkowski
symmetrizations. They found a function $n_0: (0,1)\to \N$ satisfying:
\begin{thm} [Bourgain, Lindenstrauss, Milman] \label{thm-Intro-BLM}
Let $\varepsilon \in(0,1)$ and let $n\in \N$ such that
$n_0(\varepsilon)
\le n$. If $K\in \K^n$ is a convex body with mean width $M^*(K)=M_0$,
then there exist $cn\left(C(\varepsilon) + \log n\right)$ Minkowski
symmetrizations transforming $K$ into a (convex) body $\tilde{K}$,
such that
\[
(1-\varepsilon) M_0 D_n \subset
\tilde{K} \subset
(1+\varepsilon) M_0 D_n.
\]
Here $c$ is some positive constant, and the functions
$C(\varepsilon)$, $n_0(\varepsilon)$ are of the order % $exp(c\varepsilon^{-2}\log \frac{1}{\varepsilon})$
$\left(\frac{1}{\varepsilon}\right)^\frac{c}{\varepsilon^2}$.
\end{thm}

In \cite{K_Rate} Klartag improved Theorem \ref{thm-Intro-BLM}, and
also removed the restriction $n_0\le n$, thus providing the first
truely isometric result
in all dimensions:
\begin{thm}[Klartag] \label{thm-Klartag}
Let $n\ge 2$ and $\varepsilon\in(0,1/2)$. If $K\in \K^n$ is a convex
set with mean width $M^*(K)=M_0$, then there exist $cn |\log
(\varepsilon)|$ Minkowski symmetrizations transforming $K$ into
$\tilde{K}$, such that
\[
(1-\varepsilon)M_0 D_n
\subseteq \tilde{K} \subseteq
(1+\varepsilon)M_0 D_n,
\] where $c$ is some universal constant.
\end{thm}

In this note we extend Klartag's theorem to $\Sn$. Namely, we show the following:

\begin{thm}\label{Thm_Mink-Isometric}
Let $n\ge 2$ and $\varepsilon\in(0,1/2)$. If $K\in \Sn$ is a star
shaped set with mean width $M^*(K)=M_0$, then there exist $Cn
|\log(\varepsilon)|$
Minkowski symmetrizations transforming $K$ into $\tilde{K}$, such that
\[
(1-\varepsilon)M_0 D_n
\subseteq \tilde{K} \subseteq
(1+\varepsilon)M_0 D_n,
\] where $C$ is some universal constant.
\end{thm}

The proof of Theorem \ref{Thm_Mink-Isometric} consists of three steps.
First, we make sure that the body at hand contains a small ball around
the origin (Lemma \ref{Lem_Give-Small-Ball}%, see also Remark \ref{Rem_Alternative-Small-Ball}
). Next, we consider Minkowski symmetrizations {\em of the convex
	hull} of our star shaped body. As mentioned above, for convex bodies
Klartag showed how many steps are required in order to bring a body
isometrically close to the Euclidean ball. We apply these
symmetrizations, and get a body whose convex hull lies between two
balls of very similar radii. This can only happen if the body
contains some ``$\varepsilon$-net'' of the inner ball's sphere (Lemma
\ref{Lem_Conv=Ball-to-eps-net}). In the third and final step, we use
this fact to increase the radius of the small ball, which was
obtained in the first step.

\noindent {\bf Notations:}
The {\em support function} of a (not necessarily convex) body $K$ is
defined by $h_K(u) = \sup \{ \iprod{x}{u} \ | \  x \in  K \}$. The
width, or {\em mean width}, of a star shaped set $K$ is defined to be
$M^*(K)= \int_{S^{n-1}} h_K d\sigma$, where $\sigma$ is the
normalized Haar measure on the sphere.

\section{Proof of the Theorem %\ref{Thm_Mink-Isometric}
}\label{Sec_Mink}

Our first step is to generate (using Minkowski symmetrizations) a
small ball inside a (non trivial) star shaped set.

\begin{lem}\label{Lem_Give-Small-Ball}
Let $n\ge 2$, $K_0\in \Sn$, and let $M_0=M^*(K_0)>0$. Then there exist
$c_1n$ Minkowski symmetrizations transforming $K_0$ into $K_1$, such that
\[
%c_1\left(\frac{M^*([0,Re_1])}{M^*(K_0)}\right) D_n
\frac{c_2}{\sqrt{n}} D_n
\subseteq \frac{K_1}{M_0}
%\subseteq \frac{R}{M_0} D_n
,\]
where $c_1,c_2$ are some universal constants (in fact $c_1=c$ of
Theorem \ref{thm-Klartag}).
\end{lem}
\begin{proof}
Let $R\ge M_0>0$ be the minimal radius of a centered ball enclosing
$K_0$. Then there exists some $u\in S^{n-1}$ such that $I_0=[0,Ru]
\subseteq K_0$. Let $\varepsilon = 1/e$. By Theorem \ref{thm-Klartag}
applied to $I_0$, there exist $N_1 = cn$ %$N_1 = \lceil cn \rceil$
Minkowski symmetrizations
$M_{u_1}\dots M_{u_{N_1}}$ which transform the interval $I_0$ into a
convex body $I_1$ satisfying
\[
(1-1/e) M^*(I_0) D_n \subseteq I_1.
\]
%for some universal constant $c$.
Then
\[
(1-1/e) M^*([0,u]) D_n
\subseteq
(1-1/e) \left(\frac{RM^*([0,u])}{M_0}\right) D_n
\subseteq
\frac{I_1}{M_0}
\subseteq
\frac{K_1}{M_0},
\]
where the body $K_1$ is defined by:
\begin{equation}\label{Eq_Mink-1st-Step}
K_1:=M_{u_{N_1}}\dots M_{u_1} K_0.
\end{equation}
Since $M^*([0,u])=
\frac{1}{2}\int_{S^{n-1}}|x_1|d\sigma(x)\approx
\frac{1}{\sqrt{2\pi n}}$, the proof is complete.
\end{proof}
%%
%% Start of Alternative-small-ball-computation
%%
%% TODO - Quantify below, then add to paper:
%%
%%
%\begin{rem}\label{Rem_Alternative-Small-Ball}
%An alternative way to obtain some small ball inside an image of $K_0$
%is to choose an orthonormal basis $\{e_1,\dots,e_n\}$ such that $u=
%\frac{1}{\sqrt{n}}(1,\dots,1)^T$. Apply $n$ consecutive Minkowski
%symmetrizations in directions $e_1,\dots,e_n$ to get $Q_1 = M_{e_1}
%I_0$, and $Q_{m+1} = M_{e_{m+1}}Q_m$ for $m=1,\dots,n-1$. We get the body
%\[\tilde{I_1} = RQ_n = \frac{R}{2^n}
%\sum_{\varepsilon_1,\dots,\varepsilon_n=\pm1}
%[0,(\varepsilon_1,\dots,\varepsilon_n)^T].\]
%Obviously, $\tilde{I_1}$ contains a ball of radius $R2^{-(n-1)}$. 
%\end{rem}
%% End of Alternative-small-ball-computation
%%
%%

The inner radius $c_2/\sqrt{n}$ does not decrease under additional
symmetrizations (it may increase). Next consider Minkowski
symmetrizations which bring {\em the convex hull} of $K_1$
isometrically close to the ball. By Theorem \ref{thm-Klartag} there
exist $N_2= c n |\log(\varepsilon)|$
% $N_2=\lceil c n |\log(\varepsilon)|\rceil$
directions $u_1,\dots,u_{N_2}$
such that
\[
(1 - \varepsilon) D_n
\subseteq
M_{u_{N_2}}\dots M_{u_1} \conv \left( \frac{K_1}{M_0} \right)
\subseteq
(1 + \varepsilon) D_n.
\]
We define the body $K_2$ to be:
\begin{equation}\label{Eq_Mink-2nd-Step}
K_2:=M_{u_{N_2}}\dots M_{u_1} K_1.
\end{equation}

Note that Minkowski symmetrizations commute with the convex hull
operation, i.e. $M_u\conv K = \conv M_u K$, simply because in general
$\conv (A+B) = \conv A + \conv B$. Thus $\frac{K_2}{M_0}$ is a star
shaped body whose convex hull is isometrically close to the ball. We
use the following standard lemma to show that such a body must
contain some $\delta$-net of the sphere. More precisely:
\begin{lem}\label{Lem_Conv=Ball-to-eps-net}
Let $n\ge 2$, $\varepsilon\in(0,1)$, and let $K\in \Sn$ be such that
\[(1 - \varepsilon) D_n
  \subseteq  \conv K  \subseteq
  (1 + \varepsilon) D_n.
\]
Then $K$ contains a $2\sqrt{\varepsilon}$-net of the sphere
$(1-\varepsilon) S^{n-1}$, that is
\begin{equation}\label{Eq_Delta-net}
(1-\varepsilon) S^{n-1}
\subseteq
K + 2\sqrt{\varepsilon}D_n.
\end{equation}
\end{lem}

\begin{proof}
Let $x\in (1-\varepsilon) S^{n-1}$. We claim that the intersection
$ (x+2\sqrt{\varepsilon}D_n)\cap K $ is not empty. Denoting the
hyperplane supporting $(1-\varepsilon) D_n$ at the point $x$ by $H$
and the halfspace with boundary $H$ by $H^+= \{y\,|\, % (1-\varepsilon)^2 =
\iprod{x}{x} \le \iprod{x}{y} \}$, we will show that $H^+\cap K \cap
(x+2\sqrt{\varepsilon}D_n)$ is not empty. Indeed,
\[
H^+ \cap K
\subseteq
H^+ \cap (1+\varepsilon) S^{n-1}
\subseteq
(x + 2\sqrt{\varepsilon}D_n),
\]
see Figure \ref{Fig_Instead-of-Proving}. The set $H^+ \cap K$ is not
empty (since $x\in (H^+ \cap\conv K)$), and thus the proof is complete.
\begin{figure}[h]
	\centering
	\includegraphics[width=0.35\textwidth]{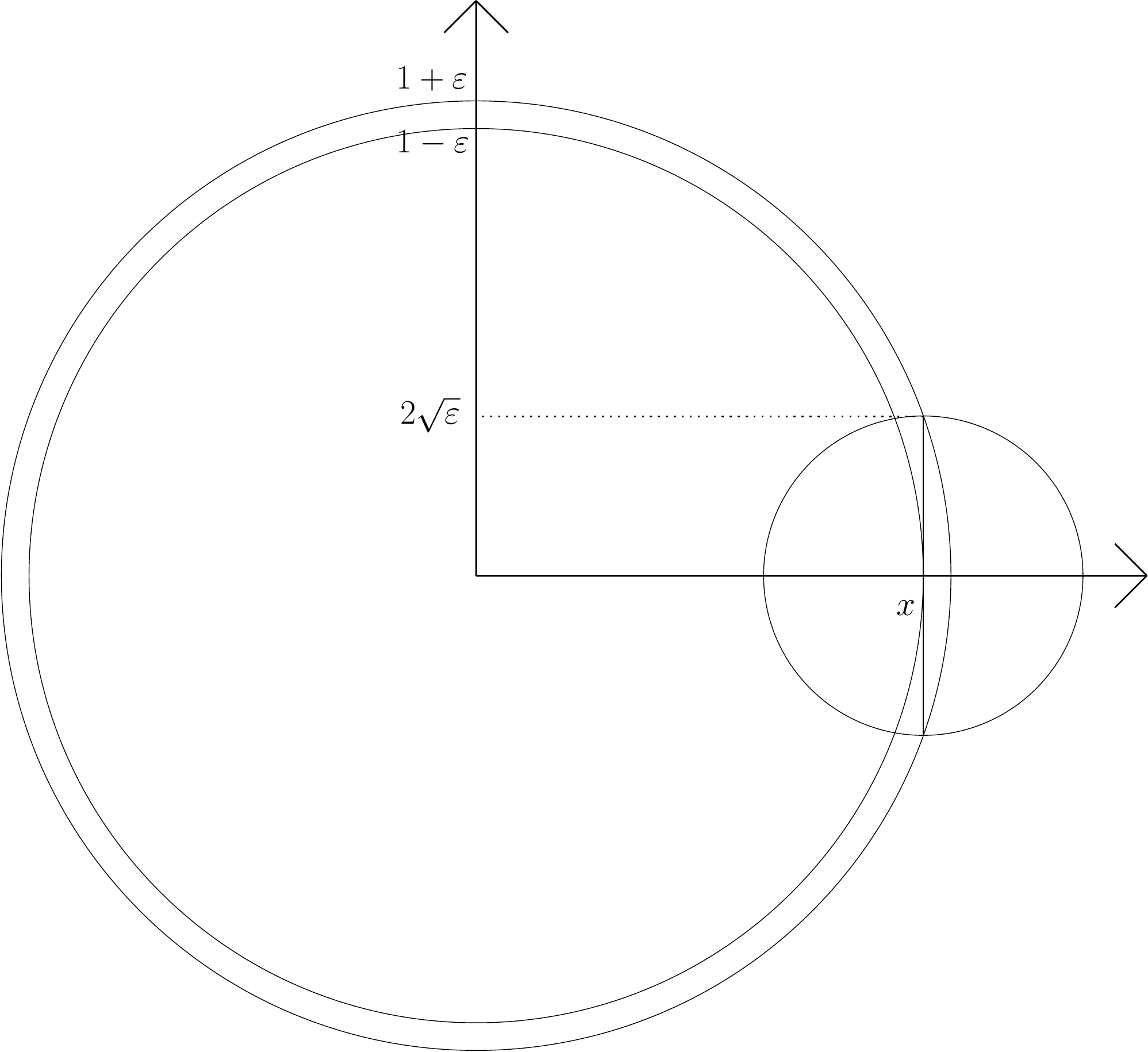}
	\caption{The small ball around $x$ contains all points of $K$ in $H^+$.}
	\label{Fig_Instead-of-Proving}
%	\label{Fig_Conv=Ball-to-eps-net}
\end{figure}
\end{proof}

Note that in fact, since $K + 2\sqrt{\varepsilon}D_n$ is star shaped, we have:
\begin{equation}
(1-\varepsilon) D_n
\subseteq
K + 2\sqrt{\varepsilon}D_n.
\end{equation}
The outer radius $1+\varepsilon$ does not increase under additional
symmetrizations (it may decrease). As for the inner radius, we begin
with the small ball obtained in Lemma \ref{Lem_Give-Small-Ball}, and
use \eqref{Eq_Delta-net} to increase it geometrically. More precisely:

\begin{lem} \label{lem-StarShapedContainsBall}
Let $n\ge 2$, $K\in \Sn$, $r\in(0,1)$, and $\varepsilon\in
(0,\varepsilon_0)$, where $\varepsilon_0=1/25$. Assume that:
\begin{itemize}
\item $r D_n \subseteq K$.
\item $(1-\varepsilon) D_n
\subseteq
K + 2\sqrt{\varepsilon}D_n$.
\end{itemize}
Then there exist $N=\alpha + \beta|\log \varepsilon| + \gamma|\log r|$ Minkowski symmetrizations transforming $K$
into $\tilde{K}$ satisfying
\[
(1-4\sqrt{\varepsilon}) D_n \subseteq \tilde{K},
\]
where $\alpha, \beta, \gamma$ are positive constants.
\end{lem}
\begin{proof}
%We first assume that $\varepsilon\le \varepsilon_0$, and
In each of the two cases $r < 2\sqrt{\varepsilon}$ and
$2\sqrt{\varepsilon} \le r$, we argue a bit differently, so we handle
them separately. We begin with the case of smaller initial inner radius.

\noindent{\bf Case a. Increasing $r$ geometrically to reach $2\sqrt{\varepsilon}$:}

If $r < 2\sqrt{\varepsilon}$ we may take the second assumption and
write for any $u\in S^{n-1}$:
\[
r\frac{1-\varepsilon}{2\sqrt{\varepsilon}} D_n
\subseteq
\frac{r}{2\sqrt{\varepsilon}}K + r D_n
\subset
K + r D_n
\subseteq
K + R_u (K),
\]
So that $r\frac{1-\varepsilon}{4\sqrt{\varepsilon}} D_n
\subseteq M_u (K)$. Defining the (decreasing) function $q:(0,1)\to\R^+$
by $q(\varepsilon)= \frac{1-\varepsilon} {4\sqrt{\varepsilon}}$, we
have $q(\varepsilon)\ge q(\varepsilon_0)= 6/5$, so the inner radius
multiplies by at least $6/5$. If $2\sqrt{\varepsilon} \le r(6/5)^m$,
then after $m$ symmetrizations the inner radius reaches
$2\sqrt{\varepsilon}$.
Thus after $N_a = 4 + 3 \left|\log\left( \frac{\varepsilon}{r^2} \right)\right|$
symmetrizations, we have reduced to the second case, where
$2\sqrt{\varepsilon} \le r$.

\noindent{\bf Case b. Increasing $r$ geometrically  towards $1$:}

Again, for any $u\in S^{n-1}$ we have
\[
\frac{(1-[2\sqrt{\varepsilon} + \varepsilon]) + r}{2} D_n
=
\frac{(1 - \varepsilon) D_n +
(r - 2\sqrt{\varepsilon}) D_n}{2}
\subseteq
\]
\[
\frac{K + 2\sqrt{\varepsilon}D_n +
(r - 2\sqrt{\varepsilon}) D_n}{2} =
\frac{K + r  D_n}{2} 
\subseteq
M_u (K).
\]
So the difference $(1-[2\sqrt{\varepsilon} + \varepsilon])- r$
decreases by at least half. The inner radius exceeds $1 -
4\sqrt{\varepsilon}$ if and only if its difference from 
$(1-[2\sqrt{\varepsilon} + \varepsilon])$ decreases below
$2\sqrt{\varepsilon} - \varepsilon > \sqrt{\varepsilon}$. Thus it
suffices to decrease that difference to $\sqrt{\varepsilon}$. For that
we require no more than $N_b=|\log_2(\sqrt{\varepsilon})| = 
\frac{1}{\log 4} |\log(\varepsilon)|$ symmetrizations.
\end{proof}

\begin{proof}[{\bf Proof of Theorem \ref{Thm_Mink-Isometric}.}] To
complete the proof one has to combine the steps above. Let $K\in\Sn$,
such that $M^\ast(K) = 1$. By Lemma \ref{Lem_Give-Small-Ball}, there
exist $cn$ Minkowski symmetrizations that transform $K$ into a set
$K_1$ such that
\[
\frac{c_2}{\sqrt{n}}D_n \subset K_1
\]
and no Minkowski symmetrization can change this fact. Let $0 <
\varepsilon < 1/25$. By Theorem \ref{thm-Klartag}, there exist
$cn|\log\varepsilon|$ symmetrization that transform $\conv(K_1)$ into
a convex set $L$ such that $(1-\varepsilon)D_n \subset L \subset (1 +
\varepsilon)D_n$. Since $M_u \conv K = \conv M_u K$, we may apply the
same symmetrizations to $K_1$ to obtain a new set $K_2$ such that
$\conv(K_2) = L$. By Lemma \ref{Lem_Conv=Ball-to-eps-net}, we get the
following:
\[
(1-\varepsilon)S^{n-1} \subset K_2 + 2\sqrt{\varepsilon}D_n.
\]
In addition, 
\[
\frac{c_2}{\sqrt{n}}D_n \subset K_2.
\]
Thus, by Lemma \ref{lem-StarShapedContainsBall} there exist $\alpha +
\beta|\log\varepsilon| + \gamma\log n$ Minkowski symmetrizations that
transform $K_2$ into $K_3$ such that
\[
(1-4\sqrt{\varepsilon})D_n \subset K_3.
\]
Recall that $K_3 \subset (1+\varepsilon)D_n$. To sum it up, we
applied no more than
\[
cn + cn|\log\varepsilon| + \alpha + \beta|\log\varepsilon| +
\gamma\log n \le C n |\log\varepsilon|
\]
symmetrizations, for some universal constant $C > 0$.

During the proof we assumed that $\varepsilon < 1/25$. This can be
easily changed to $\varepsilon < 1/2$, at the cost of a different
constant in the expression $C n |\log\varepsilon|$, by always
symmetrizing the set to be $\varepsilon / 25$ close to the Euclidean
ball. By the same argument one may extend for all $\varepsilon\in
(0,1)$, and the corresponding bound on the number of symmetrizations
will become $n (C|\log\varepsilon| + C')$.
\end{proof}

\noindent Dan Itzhak Florentin, danflorentin@gmail.com\\
\noindent Department of Mathematics, Weizmann Institute of Science, Rehovot, Israel.\\

\noindent Alexander Segal, segalale@gmail.com\\
\noindent School of Mathematical Science, Tel Aviv University, Tel Aviv, Israel.\\

\end{document}